\documentclass{amsart}
\usepackage{amsfonts,amsmath,color,graphicx,calligra,mathrsfs,tikz,tikz-cd,mathtools}
\usepackage{amsthm,amssymb}

\usetikzlibrary{arrows}

\usepackage[english]{babel}
\usepackage{comment}
\usepackage{hyperref}

\newtheorem{thm}{Theorem}
\newtheorem{cor}[thm]{Corollary}

\newtheorem{prop}[thm]{Proposition}
\theoremstyle{definition}

\theoremstyle{remark}

\theoremstyle{definition}

\theoremstyle{definition}

\theoremstyle{remark}

\numberwithin{equation}{section}
\newcommand{\aaa}{\mathfrak{a}}
\newcommand{\bbb}{\mathfrak{b}}

\newcommand{\mmm}{\mathfrak{m}}
\newcommand{\nnn}{\mathfrak{n}}

\newcommand{\locala}{(A,\mathfrak{m},k)}
\newcommand{\localb}{(B,\mathfrak{n},l)}

\newcommand{\Tor}{\mathrm{Tor}}

\newcommand{\HH}{\mathrm{H}}

\begin{document}

\title[On a theorem of Gulliksen on the homology of local rings]{On a theorem of Gulliksen on the homology of local rings}
\author{Samuel Alvite, Nerea G. Barral and Javier Majadas}%
\address{Departamento de Matem\'aticas, Facultad de Matem\'aticas, Universidad de Santiago de Compostela, E15782 Santiago de Compostela, Spain}%
\email{ samuelalvite@gmail.com, nereabarral@gmail.com, j.majadas@usc.es}%

\thanks{$^{(\star)}$ This work was partially supported by Agencia Estatal de Investigaci\'on (Spain), grant MTM2016-79661-P (European FEDER support included, UE) and by Xunta de Galicia through the Competitive Reference Groups (GRC) ED431C 2019/10}


\begin{abstract}
We show that a modification of the proof of a result of Gulliksen gives an elementary proof of the following important theorem by Avramov: if $f:\locala \to \localb$ is a homomorphism of noetherian local rings and $B$ is of finite flat dimension over $A$, then the homomorphism induced in André-Quillen homology modules $\HH_2(A,l,l)\to \HH_2(B,l,l)$ is injective.

\end{abstract}

\maketitle

\setcounter{section}{-1}

In 1969 Gulliksen showed that in a noetherian local ring, an ideal of finite projective dimension with free first Koszul homology module is generated by a regular sequence \cite[1.4.9]{GL}. His proof is elementary, using only basic notions of differential graded commutative algebras.\\

In \cite{Ro}, Rodicio showed that the idea of Gulliksen's proof can also be used to give the following result: if $\aaa$ is an ideal of a noetherian local ring $\locala$ such that its associated first Koszul homology module is $A /\aaa$-free, then the canonical homomorphism in Andr\'e-Quillen homology (\cite{An1974}, \cite{QMIT}) $\HH_3(A/\aaa,k,k)\to \HH_2(A,A/\aaa,k)$ vanishes. As a consequence he obtains generalizations of the above Gulliksen's result as well as another result of André on pairs of complete intersections.\\

In between these papers, Avramov proved the following apparently unrelated result: if $f:\locala \to \localb$ is a homomorphism of noetherian local rings and $B$ is of finite flat dimension over $A$, then the homomorphism $\HH_2(A,l,l)\to \HH_2(B,l,l)$ is injective (first in the flat case in \cite{Av1975}, and then in full generality, even with higher dimensional analogues in \cite{Av1982}). This result has many interesting applications; for instance it allowed Avramov himself to prove that the complete intersection property localizes.\\

 We follow the path in \cite{Ro} further exploiting Gulliksen's idea. We show here that a slight modification of Gulliksen's proof gives a result (Corollary \ref{ga}) that contains Avramov's theorem as well as Gulliksen's one (Corollary \ref{ga} in the case $A=B$, $\aaa=\bbb$) and Rodicio's (as his proof \cite[p. 61]{Ro} shows). But more important than unifying these three results is the fact that the proof is as elementary as Gulliksen's one and gives also Avramov's theorem. Avramov's original proof, even in the flat case, is not so elementary since it uses the existence of \emph{minimal} differential algebra resolutions, and since 1975 this was the only existing proof. In order to prove Avramov's theorem, taking a Cohen factorization \cite{AFH}, it is easy to restrict ourselves to the case when $B=A/\aaa$, which is how we present his result (Corollary \ref{Avramov}).\\

\hfill

A minimal set of generators of an ideal $\aaa$ of a (not necessarily noetherian) local ring $\locala$ is a set of generators such that their images in $\aaa/\mmm\aaa$ form a basis of this $k$ vector space.

\begin{prop}
Let $f:\locala \to \localb$ be a homomorphism of local rings, $\aaa$ an ideal of $A$, $\bbb$ an ideal of $B$ such that $f(\aaa)\subset \bbb$. Assume that $\aaa$ has a minimal set of generators $\{x_i\}$. Let $\{f(x_i)\}\cup \{y_j\}$ be a set of generators of the ideal $\bbb$. Let $E$ be the Koszul complex associated to the set of generators $\{x_i\}$ of the ideal $\aaa$ and $M$ the Koszul complex associated to the set of generators $\{f(x_i)\}\cup \{y_j\}$ of the ideal $\bbb$, so we have a canonical homomorphism
\[ \bar{\alpha}: \HH_1(E)\otimes_{A/\aaa}B/\bbb \to \HH_1(M). \]
If there exists some $r>0$ such that $\Tor_{2r}^A(A/\aaa,k)=0$ then for any homomorphism of $B/\bbb$-modules $\varphi: \HH_1(M) \to B/\bbb$ we have
\[ Im(\varphi \circ \bar{\alpha} )\subset \nnn/\bbb. \]
\end{prop}
\begin{proof}
Let $E=A<X_i;\; dX_i=x_i>$ (notation as in \cite{GL}), $M=B<X_i,Y_j;\; dX_i=f(x_i),\; dY_j=y_j>$. Let $\{s_u\}$ be a set of cycles in $E_1$ such that their classes $\{\overline{s_u}\}$ in $\HH_1(E)$ generate $\HH_1(E)$. Since $\{x_i\}$ is a minimal set of generators of $\aaa$, we have $s_u\in \mmm E_1$ for all $u$.

Let $\varphi: \HH_1(M) \to B/\bbb$ be a homomorphism of $B/\bbb$-modules. Let $\beta_u\in B$ be elements whose classes in $B/\bbb$ are $\varphi \bar{\alpha} (\overline{s_u}\otimes 1)$. Let $G$ be a DG resolution of the $A$-algebra $A/\aaa$ with 1-skeleton $E$ (that is, with the notation of \cite[p.11]{GL}, $\mathrm{F}_1G=E$) and 2-skeleton $F_2G=E<S_u;\; dS_u=s_u>$. Let $P$ be a DG resolution of the $B$-algebra $B/\bbb$ with 2-skeleton $N=B<X_i,Y_j,S_u,T_v;\; dT_v=t_v>$ where $\mathrm{deg}(T_v)=2$ and $\{\overline{t_v}\}$ generates $\HH_1(M<S_u;\;dS_u=\alpha(s_u)>)$ and containing $G\otimes_A B$ (by \cite[Theorem 1.2.3]{GL} there exists such $P$).

By \cite[Lemma 1.3.4]{GL} there exists a derivation $D:N\to N$ of degree -2 such that $D(S_u)=\beta_u$, $D(T_v)=\gamma_v$, $D(M)=0$, where $\gamma_v \in B$ represents $\varphi (\widetilde{t_v})\in B/\bbb$ (where $\widetilde{t_v}\in \HH_1(M)$ is the class of $t_v\in M$). For a complex $X$ we will denote by $Z_iX\subset X_i$ the submodule of degree $i$ cycles and similarly the boundaries by $B_iX\subset X_i$. The diagram
\[
\begin{tikzcd}
N_2 \arrow[rr, "D" ]\arrow[d, "d" ]& & N_0=B \arrow[d]\\
Z_1N=Z_1M \arrow[r] & \HH_1(M) \arrow[r, "\varphi"] & B/\bbb
\end{tikzcd}
\]
is then commutative. In particular, $D(Z_2N)\subset \bbb =B_0M =B_0N$, and therefore we can extend $D$ to the 3-skeleton $\mathrm{F}_3P$ of $P$. Inductively, once we have $D:\mathrm{F}_iP\to \mathrm{F}_iP$, since $D(Z_i(\mathrm{F}_iP))\subset Z_{i-2}(\mathrm{F}_iP)=B_{i-2}(\mathrm{F}_iP)$, we can extend $D$ to $\mathrm{F}_{i+1}P$, and then we have
\[ D:P\to P. \]

Since $s_u\in \mmm E_1$, we have $0=s_u\otimes 1\in G\otimes_Ak$, and then $S_u^{(r)}\otimes 1\in Z_{2r}(G\otimes_Ak)$, where $S_u^{(r)}$ is the $r$-th divided power of $S_u$. The hypothesis $\Tor_{2r}^A(A/\aaa,k)=0$ implies $B_{2r}(G\otimes_Ak)=Z_{2r}(G\otimes_Ak)$ and then the image of $S_u^{(r)}\otimes 1\in B_{2r}(G\otimes_Ak)$ in $G\otimes_Al$ belongs to $B_{2r}(G\otimes_Al)$. We have
\[ \beta_u^r\otimes 1=(D^r\otimes 1)(S_u^{(r)}\otimes 1)\in (D^r\otimes 1)(B_{2r}(G\otimes_Al))\subset \]
\[ (D^r\otimes 1)(B_{2r}(P\otimes_Bl))\subset B_0(P\otimes _Bl)=B_0(M\otimes_Bl)=0. \]
Then $\beta_u^r\in\nnn/\bbb$ and therefore $\beta_u\in\nnn/\bbb$ as desired.

\end{proof}

\begin{cor}\label{ga}
Let $f:\locala \to \localb$ be a homomorphism of noetherian local rings, $\aaa$ an ideal of $A$, $\bbb$ an ideal of $B$ such that $f(\aaa)\subset \bbb$. If $\mathrm{pd}_A(A/\aaa)<\infty$ and the first Koszul homology module associated to a set of generators of $\bbb$ is a free $B/\bbb$-module, then the canonical map
\[  \HH_2(A,A/\aaa,l)\to \HH_2(B,B/\bbb,l)  \]
vanishes.

\end{cor}
\begin{proof}
Let $E$, $M$ and $\bar{\alpha}: \HH_1(E)\otimes_{A/\aaa}B/\bbb \to \HH_1(M)$ be as in the Proposition. We have $\mathrm{Im}(\bar{\alpha})\subset \nnn\HH_1(M)$ (if $\bar{\alpha}(x)\in \HH_1(M)-\nnn\HH_1(M)$ then $\bar{\alpha}(x)$ is part of a basis of the free $B/\bbb$-module $\HH_1(M)$ and so there exists $\varphi:\HH_1(M)\to B/\bbb$ such that $\varphi\bar{\alpha} (x)=1\notin \nnn/\bbb)$. Therefore
\[  \HH_1(E)\otimes_{A/\aaa}l\to \HH_1(M)\otimes_{B/\bbb}l   \]
vanishes, and then the result follows from \cite[15.12]{An1974}.
\end{proof}

\begin{cor}\label{Avramov} (Avramov, \cite{Av1975}, \cite{Av1982})
Let $A$ be a noetherian local ring, $\aaa\neq 0$ an ideal of $A$. If $\mathrm{pd}_A(A/\aaa)<\infty$ then the canonical map
\[  \HH_2(A,k,k)\to \HH_2(A/\aaa,k,k)  \]
is injective.
\end{cor}
\begin{proof}
It follows from Corollary \ref{ga} and the Jacobi-Zariski exact sequence
\[  \HH_2(A,A/\aaa,k)\to \HH_2(A,k,k)\to \HH_2(A/\aaa,k,k).  \]
\end{proof}

\end{document}